\documentclass[12pt]{amsart}
\usepackage{color}
\usepackage{graphicx}
\usepackage{subfigure}
\usepackage{amssymb}
\usepackage{amsmath}
\usepackage{multirow}
\usepackage{enumerate}
\usepackage{epstopdf}
\usepackage{cite,stmaryrd,txfonts}
\usepackage{tikz}

\input{undertilde}
\usetikzlibrary{snakes}

\usepackage{verbatim}

\usepackage{epic}

\usepackage{empheq}

\usepackage{ulem}

\newcommand{\bs}{\boldsymbol}

\def\cequiv{\raisebox{-1.5mm}{$\;\stackrel{\raisebox{-3.9mm}{=}}{{\sim}}\;$}}

\def\utau{\undertilde{\tau}}
\def\uphi{\undertilde{\varphi}}
\def\upsi{\undertilde{\psi}}
\def\ueta{\undertilde{\eta}}
\def\uzeta{\undertilde{\zeta}}

\def\uf{\undertilde{f}}

\def\uq{\undertilde{q}}
\def\uu{\undertilde{u}}
\def\uw{\undertilde{w}}
\def\uv{\undertilde{v}}
\def\ux{\undertilde{x}}

\def\curl{{\rm curl}}
\def\dv{{\rm div}}

\def\vgm12{\bs{V}^{1+,2}_{\gamma,M}}

\newtheorem{theorem}{Theorem}
\newtheorem{remark}[theorem]{Remark}
\newtheorem{proposition}[theorem]{Proposition}
\newtheorem{lemma}[theorem]{Lemma}

\newtheorem{definition}[theorem]{Definition}

\newcounter{mnote}
\let\oldmarginpar\marginpar
\renewcommand\marginpar[1]{\-\oldmarginpar[\raggedleft\footnotesize #1]%
  {\raggedright\footnotesize #1}}

\begin{document}

\title{Decoupled mixed element schemes for fourth order problems}
\author{Shuo Zhang}
\address{LSEC, Institute of Computational Mathematics and Scientific/Engineering Computing, Academy of Mathematics and System Sciences, Chinese Academy of Sciences, Beijing 100190, People's Republic of China}
\email{szhang@lsec.cc.ac.cn}
\thanks{The author is supported by the National Natural Science Foundation of China(Grant No. 11471026) and National Centre for Mathematics and Interdisciplinary Sciences, Chinese Academy of Sciences.
}



\begin{abstract}
In this paper, we study decoupled mixed element schemes for fourth order problems. A general process is designed such that an elliptic problem on high-regularity space is transformed to a decoupled system with spaces of low order involved only and is further discretised by low-degree finite elements. The process can be fit for various fourth order problems, and is used in the remaining of the paper particularly for three-dimensional bi-Laplacian equation to conduct a family of mixed element discretisation schemes.
\end{abstract}

\maketitle


%
%
%
\section{Introduction}
Fourth order problems, of which the biLaplacian equation is a representative one, fall in the fundamental model problems in applied mathematics and are also frequently encountered and dealt with in applied sciences. Their discretisation have been attracting wide research interests. As a first approach, many kinds of conforming and nonconforming finite elements for second order Sobolev spaces are designed, and the discretisation of fourth order problems in primal formulations have been discussed in wide literature. We refer to, e.g., \cite{Ciarlet.P1978} and \cite{Engel.G;Garikipati.K;Hughes.T;Larson.M;Mazzei.L;Taylor.R2002,Zenivsek.A1973,Zhang.Shangyou2009,Wang.M;Shi.Z;Xu.J2007NM,Wang.M;Shi.Z;Xu.J2007JCM,Wang.M;Xu.J2007} for instances. Particularly, the two dimensional Morley element for biharmonic equation is generalised by \cite{Wang.M;Xu.J2006} to arbitrary dimensions and then by \cite{Wang.M;Xu.J2013} to arbitrary $(-\Delta)^m$ problem in $n$ dimension with $n\geqslant m$. The elements are usually designed case by case, and are complicated especially in high dimensions.

An alternative way is to transform the primal problems to order reduced formulations. This is generally to construct a system on low-regularity spaces by introducing auxiliary variables, and discretize the generated system with numerical schemes. Various kinds of mixed methods have been designed for the biLaplacian equation; we refer to, e.g., \cite{Ciarlet.P;Raviart.P1974,Hermann.L1967,Johnson.C1973,Hellan.K1967,Behrens.E;Guzman.J2011,GiraultRaviart1986,Li.Z;Zhang.S2016,Pauly.D;Zulehner.W2016} for examples. Recently, different from the aforementioned methods, an order reduction framework based on regular decomposition was presented in \cite{Zhang.S2016fw} for general fourth order problems. Under this framework, once a ``configurated" condition is verified, a problem on high-regularity space can be transferred constructively to an equivalent system on a triple of low-regularity spaces, and the framework can be fit for various fourth order problems.

In this paper, we study the mixed element method for fourth order problems, with attention paid specially on the construction and discretisation of decoupled formulation. Actually, as auxiliary variables are introduced, the generated system will be of larger size; moreover, the variables can possess different regularities and thus possess different capacities of approximation accuracy with respect to finite element spaces. A decoupled formulation may provide opportunities for smaller systems and more flexible finite element choices. Relevant to the framework of \cite{Zhang.S2016fw}, in this paper, a constructive process is presented on decoupling the primal problem to three subsystems and then discretizing them with finite elements with lower regularity. The three subsystems contain two elliptic ones and a saddle-point problem. The existence and well-posedness of the decoupled form can be proved without extra assumption. Though, the well-posed saddle-point subproblem can be replaced by a singular one without loss of equivalence. The admission of singular problem can bring in convenience; this will be discussed more in future works. The validity of the methodology is verified with the bi-Laplacian equation in three dimensional, which is a fundamental model problem arising in, e.g., the linear elasticity model in the formulation of Galerkin vector (c.f. \cite{Gurtin.M1973}) and in the transmission eigenvalue problem (c.f. \cite{Colton.D;Monk.P1988,Kirsch.A1986}) in acoustics. Decoupled formulations and a family of discretisation schemes are presented for the model problem; optimal accuracy with respect to both the regularity of the solution and the degree of finite elements are obtained. 

There have been several fruitful and powerful approachs for solving fourth order problems which make use of kinds of stabilizations, such as the discontinuous Galerkin(dG) method (c.f., e.g.,\cite{Brenner.S;Monk.P;Sun.J2015,Brenner.S;Sung.L2005,Georgoulis.E;Houston.P2009}) and the weak Galerkin(wG) method (c.f., e.g., \cite{Mu.L;Wang.J;Ye.X;Zhang.S2014,Zhang.R;Zhai.Q2015,Wang.C;Wang.J2014}) for biLaplacian equations. They methods use mesh-dependent bilinear forms, and will not be discussed in this paper. A comparison between them and the method in this paper could be an interesting topic. 

The remaining of the paper is organised as follows. Some notation is given in the remaining part of this section. In Section \ref{sec:abs}, a general decomposition process is presented as a framework. In Sections \ref{sec:mixbl} and \ref{sec:fem}, continuous and discretised mixed formulations of the bi-Laplacian equation are studied under the framework.  Finally, some concluding remarks are given in Section \ref{sec:con}. 

In this paper, we apply these notation. Let $\Omega\subset\mathbb{R}^3$ be a simply connected polyhedron domain, and $\Gamma=\partial\Omega$ be the boundary of $\Omega$. We use $H^2_0(\Omega)$, $H^1_0(\Omega)$, $H_0(\curl,\Omega)$ and $H_0(\dv,\Omega)$ for certain Sobolev spaces as usual, and specifically, denote  $\displaystyle L^2_0(\Omega):=\{w\in L^2(\Omega):\int_\Omega w dx=0\}$, $\undertilde{H}{}^1_0(\Omega):=(H^1_0(\Omega))^3$, $N_0(\curl,\Omega):=\{\ueta\in H_0(\curl,\Omega), (\ueta,\nabla s)=0\ \forall\,s\in H^1_0(\Omega)\}$, and $\mathring{H}_0(\dv,\Omega):=\{\utau\in H_0(\dv,\Omega):\dv\utau=0\}$.  We use $``\undertilde{~}"$ for vector valued quantities in the present paper. We use $(\cdot,\cdot)$ for $L^2$ inner product and $\langle\cdot,\cdot\rangle$ for the duality between a space and its dual. Occasionally, $\langle\cdot,\cdot\rangle$ can be treated as $L^2$ inner product without ambiguity. Finally, $\lesssim$, $\gtrsim$, and $\cequiv$ respectively denote $\leqslant$, $\geqslant$, and $=$ up to a constant. The hidden constants depend on the domain, and, when triangulation is involved, they also depend on the shape-regularity of the triangulation, but they do not depend on $h$ or any other mesh parameter.

\section{Decoupled order reduced formulation of high order problem}
\label{sec:abs}

We consider the variational problem: given $f\in V'$, find $u\in V$, such that
\begin{equation}\label{eq:primal}
a(u,v)=\langle f,v\rangle,\quad\forall\,v\in V,
\end{equation}
where $V$ be a Hilbert space, and $a(\cdot,\cdot)$ is equivalently an inner product on $V$. In many applications, $V$ can be some Sobolev space of high order. We thus discuss its order reduced formulations.

\subsection{Configurable triple and order deduced formulation }

\begin{definition}(\cite{Zhang.S2016fw})
Given two Hilbert spaces $R$ and $Y$ and an operator $B$, if there is a Hilbert space $H$, such that $Y\subset H$ continuously, and $B$ is a closed operator that maps $R$ into $H$, then the triple $\{R,Y,B\}$ is called a \textbf{configurable} triple, and $H$ is called one ground space of $\{R,Y,B\}$. Given a Hilbert space $W$ equipped with norm $\|\cdot\|_W$, if $W=\{w\in R:Bw\in Y\}$, and $\|w\|_W$ is equivalent to $\|w\|_R+\|Bw\|_Y$, then $W$ is called to be \textbf{configurated} by the triple $\{R,Y,B\}$.
\end{definition}
\begin{lemma}(\cite{Zhang.S2016fw})\label{lem:equinorm}
For $\{R,Y,B\}$ a configurable triple with $H$ the ground space, define $\Sigma:=BR+Y$. Then $\Sigma$ is a Hilbert space with respect to the norm $\|\cdot\|_\Sigma$ defined by
$$
\|\sigma\|_\Sigma:=\inf_{r\in R,y\in Y,Br+y=\sigma}\|Br\|_H+\|y\|_Y.
$$
Moreover, $BR$ is closed in $\Sigma$ and $\|Br\|_H\cequiv \|Br\|_\Sigma$ for $r\in R$.
\end{lemma}

In the sequel, we always set the hypothesis below.
\paragraph{\bf Hypothesis 0}
\begin{enumerate}
\item $\{R,Y,B\}$ is a configurable triple, and $V$ is configurated by the triple $\{R,Y,B\}$.
\item $a(u,v):=a_R(u,v)+b(u,Bv)+b(v,Bu)+a_Y(Bu,Bv)$ is an inner product on $V$, where $a_{R,Y}(\cdot,\cdot)$ is a bounded symmetric semidefinite bilinear form on $R,Y$, respectively, and $b(\cdot,\cdot)$ is a bounded bilinear form on $R\times Y$.
\end{enumerate}
Given $f\in V'$, it can be represented as $\langle f,v\rangle=\langle f_R,v\rangle+\langle f_Y,Bv\rangle$ for some $f_R\in R'$ and $f_Y\in Y'$. We consider a specific case of the problem \eqref{eq:primal}: find $u\in V$, such that
\begin{equation}\label{eq:vpexp}
a_R(u,v)+b(u,Bv)+b(v,Bu)+a_Y(Bu,Bv)=\langle f_R,v\rangle+\langle f_Y,Bv\rangle,\ \ \forall\,v\in V.
\end{equation}
Corresponding to \cite{Zhang.S2016fw}, an equivalent formulation of \eqref{eq:vpexp} is to find $(u,y,g_+)\in X:=R\times Y\times \Sigma'$, such that, for any $(v,z,l_+)\in X$,
\begin{equation}\label{eq:relaxtwicedual}
\left\{
\begin{array}{cccl}
\displaystyle a_R(u,v)&+b(v,y)&+\langle g_+,Bv\rangle&=\langle f_R,v\rangle,
\\
\displaystyle b(u,z)&+a_Y(y,z)&-\langle g_+,z\rangle&=\langle f_Y,z\rangle,
\\
\displaystyle \langle l_+,Bu\rangle&-\langle l_+,y\rangle&&=0.
\end{array}
\right.
\end{equation}

\begin{lemma}\label{lem:abs}
Given $f_R\in R'$ and $f_Y\in Y'$, the problem \eqref{eq:relaxtwicedual} admits a unique solution $(u,y,g_+)\in X$, and
$$
\|u\|_{W_1}+\|y\|_{W_2}+\|g_+\|_{\Sigma'}\cequiv \|f_R\|_{R'}+\|f_Y\|_{Y'}.
$$
Moreover, $u$ solves the primal problem \eqref{eq:vpexp}.
\end{lemma}

\subsection{An expanded order reduced formulation}
\label{subsec:decform}

As $BR$ is closed in $\Sigma$, the decomposition holds that $\Sigma=BR\oplus (BR)^\perp$, with $(BR)^\perp$ being the orthogonal complement of $BR$ in $\Sigma$. Meanwhile, $\Sigma'=(BR)^0\oplus((BR)^\perp)^0$. In this paper, we use the superscript $``{\cdot}^0"$ to denote the polar set. This way, the problem \eqref{eq:relaxtwicedual} can be rewritten as: find $(u,y,g^0,g_\perp^0)\in R\times Y\times (BR)^0\times ((BR)^\perp)^0$, such that, for $(v,z,l^0,l_\perp^0)\in R\times Y\times (BR)^0\times ((BR)^\perp)^0$,
\begin{equation}\label{eq:dualdecomp}
\left\{
\begin{array}{ccccl}
\displaystyle a_R(u,v)&+b(v,y)&&+\langle g_\perp^0,Bv\rangle&=\langle f_R,v\rangle,
\\
\displaystyle b(u,z)&+a_Y(y,z)&-\langle g^0,z\rangle&-\langle g_\perp^0,z\rangle&=\langle f_Y,z\rangle,
\\
\displaystyle &-\langle l^0,y\rangle&&&=0.
\\
\displaystyle \langle l_\perp^0,Bu\rangle&-\langle l_\perp^0,\phi\rangle&&&=0.
\end{array}
\right.
\end{equation}
Now we turn to the representation of $(BR)^0$ and $((BR)^\perp)^0$. Let $c(\cdot,\cdot)$ be the inner product on $\Sigma$, and $d(\cdot,\cdot)$ be a bilinear form on $\Sigma$ and satisfy the property below:
\paragraph{\bf Property I}
\begin{enumerate}
\item $d(\cdot,\cdot)$ is continuous on $\Sigma$ and $d(Bs,\tau)=0$ for any $s\in R$ and $\tau\in \Sigma$;
\item $\displaystyle\inf_{\tau\in (BR)^\perp}\sup_{\sigma\in (BR)^\perp}\frac{d(\tau,\sigma)}{\|\tau\|_\Sigma\|\sigma\|_\Sigma}\geqslant C>0$.
\end{enumerate}

\begin{lemma}\label{lem:infsupequi}
Let $d(\cdot,\cdot)$ be a continuous bilinear form (not necessarily symmetric) defined on $\Sigma$ such that $d(Bs,\tau)=0$ for any $s\in R$ and $\tau\in \Sigma$. Then
$$\displaystyle\inf_{\tau\in (BR)^\perp}\sup_{\sigma\in (BR)^\perp}\frac{d(\tau,\sigma)}{\|\tau\|_\Sigma\|\sigma\|_\Sigma}
\cequiv \displaystyle\inf_{\tau\in (BR)^\perp}\sup_{y\in Y}\frac{d(\tau,y)}{\|\tau\|_\Sigma\|y\|_Y}
\cequiv \displaystyle\inf_{\tau\in \Sigma}\sup_{y\in Y,r\in R}\frac{d(\tau,y)+c(\tau,Br)}{\|\tau\|_\Sigma(\|y\|_Y+\|Br\|_H)}.$$
\end{lemma}
\begin{proof}
Given $\sigma\in \Sigma$, there exists a $y_\sigma\in Y$ and $r_\sigma\in R$, such that $\sigma=y_\sigma+Br_\sigma$, and $\|\sigma\|_\Sigma\geqslant C(\|y_\sigma\|_Y+\|Br_\sigma\|_H)$. Meanwhile, given $y\in Y$, it can be decomposed uniquely as $y=\sigma_y+Br_y$ with $\sigma_y\in(BR)^\perp$ and thus $\|\sigma_y\|_\Sigma\leqslant \|y\|_\Sigma\leqslant \|y\|_Y$. This shows $\displaystyle\inf_{\tau\in (BR)^\perp}\sup_{\sigma\in (BR)^\perp}\frac{d(\tau,\sigma)}{\|\tau\|_\Sigma\|\sigma\|_\Sigma}
\cequiv \displaystyle\inf_{\tau\in (BR)^\perp}\sup_{y\in Y}\frac{d(\tau,y)}{\|\tau\|_\Sigma\|y\|_Y}$.
For the remaining part, it holds that
\begin{multline*}
\displaystyle\inf_{\tau\in \Sigma}\sup_{y\in Y,r\in R}\frac{d(\tau,y)+c(\tau,Br)}{\|\tau\|_\Sigma(\|y\|_Y+\|Br\|_H)} \cequiv \inf_{\tau_1\in BR,\tau_2\in (BR)^\perp}\sup_{y\in Y,r\in R}\frac{d(\tau_2,y)+c(\tau_1,Br)}{(\|\tau_1\|_\Sigma+\|\tau_2\|_\Sigma)(\|y\|_Y+\|Br\|_H)}
\\
\cequiv \min\left(\inf_{\tau_1\in BR}\sup_{r\in R}\frac{c(\tau_1,Br)}{\|\tau_1\|_\Sigma\|Br\|_H},  \inf_{\tau_2\in (BR)^\perp}\sup_{y\in Y}\frac{d(\tau_2,y)}{\|\tau_2\|_\Sigma\|y\|_Y}  \right)\cequiv \inf_{\tau_2\in (BR)^\perp}\sup_{y\in Y}\frac{d(\tau_2,y)}{\|\tau_2\|_\Sigma\|y\|_Y} .
\end{multline*}
The proof is completed. 
\end{proof}

The lemma below is evident.
\begin{lemma}\label{lem:repfnctl}
Let $d(\cdot,\cdot)$ be a bilinear form on $\Sigma$ that satisfied {\bf Property I}. Given $g\in \Sigma'$, 
\begin{enumerate}
\item $g\in ((BR)^\perp)^0$ if and only if there is an $r\in R$, such that $\langle g,\sigma\rangle=c(Br,\sigma)$ for any $\sigma\in \Sigma$;
\item $g\in (BR)^0$ if and only if there is a $\tau_g\in (BR)^\perp$, such that $\langle g,\sigma\rangle=c(\tau_g,\sigma)$ for any $\sigma\in \Sigma$;
\item $g\in (BR)^0$ if and only if there is a $\tau_g\in \Sigma$, such that $\langle g,\sigma\rangle=d(\tau_g,\sigma)$ for any $\sigma\in \Sigma$.
\end{enumerate}
\end{lemma}
An equivalent problem of \eqref{eq:dualdecomp} is then to find $(u,y,\xi,r)\in  R\times Y\times \Sigma\times R$, such that for $(v,z,\eta,s)\in  R\times Y\times \Sigma\times R$,
\begin{equation}\label{eq:redehvardualdecomp}
\left\{
\begin{array}{ccccl}
\displaystyle a_R(u,v)&+b(v,y)&&+ c(Br,Bv) &=\langle f_R,v\rangle,
\\
\displaystyle b(u,z)&+a_Y(y,z)&-d(\xi,z)&-c(Br,z)&=\langle f_Y,z\rangle,
\\
\displaystyle &-d(\eta,y)&&&=0,
\\
\displaystyle c(Bs,Bu)&-c(Bs,y)&&&=0.
\end{array}
\right.
\end{equation}
As $d(Bs,\tau)=0$ for $s\in R$ and $\tau\in \Sigma$, $\xi$ can not be uniquely determined by \eqref{eq:redehvardualdecomp}. By introducing further a Lagrangian multiplier, we consider a problem: find $(p,u,y,\xi,r)\in R\times R\times Y\times \Sigma\times R$, such that for $(q,v,z,\eta,s)\in R\times R\times Y\times \Sigma\times R$,
\begin{equation}\label{eq:ehvardualdecompctod}
\left\{
\begin{array}{cccccl}
&&&c(\xi,Bq)&&=0,
\\
&\displaystyle a_R(u,v)&+b(v,y)&&+ c(Br,Bv) &=\langle f_R,v\rangle,
\\
&\displaystyle b(u,z)&+a_Y(y,z)&-d(\xi,z)&-c(Br,z)&=\langle f_Y,z\rangle,
\\
c(\eta,Bp)&\displaystyle &-d(\eta,y)&&&=0,
\\
&\displaystyle c(Bs,Bu)&-c(Bs,y)&&&=0.
\end{array}
\right.
\end{equation}
The theorem below holds immediately. 
\begin{theorem}
Provided \textbf{Hypothesis 0}, let $c(\cdot,\cdot)$ be the inner product on $\Sigma$, and $d(\cdot,\cdot)$ be a bilinear form on $\Sigma$ that satisfies \textbf{Property I}. Given $f_R\in R'$ and $f_Y\in Y'$,
\begin{enumerate}
\item System \eqref{eq:ehvardualdecompctod} admits a unique solution $(p,u,y,\xi,r)$, and $u$ solves \eqref{eq:vpexp} and $(u,y,\xi,r)$ solves \eqref{eq:redehvardualdecomp};
\item if $(u,y,\xi,r)$ and $(\hat u,\hat y,\hat \xi,\hat r)$ are two solutions of \eqref{eq:redehvardualdecomp}, then $u=\hat u$, $\xi=\hat\xi$ and $r=\hat r$.
\end{enumerate}
\end{theorem}
\subsection{Decoupled order reduced formulations}
In the sequel, we focus ourselves on the case that the assumption below holds.
\paragraph{\bf Assumption II} $B$ is injective on $R$, $a_R(u,v)=0$ for  $u,v\in R$, and $b(v,z)=0$ for $v\in R$ and $z\in Y$.

A main result of this paper is the theorem below. 
\begin{theorem}
Provided \textbf{Hypothesis 0} and \textbf{Assumption II}, let $c(\cdot,\cdot)$ be the inner product on $\Sigma$, and $d(\cdot,\cdot)$ be a bilinear form on $\Sigma$ that satisfies \textbf{Property I}. 
\begin{enumerate}
\item The problem \eqref{eq:vpexp} can be decoupled in accordance with the problem \eqref{eq:ehvardualdecompctod} as 
\begin{enumerate}
\item find $r\in R$, such that
\begin{equation}\label{eq:5vsystem-1}
c(Br,Bv)=\langle f_R,v\rangle,\quad\forall\,v\in R;
\end{equation}
\item with $r$ solved out, find $(p,y,\xi)\in R\times Y\times \Sigma$, such that, for $(q,z,\eta)\in R\times Y\times\Sigma$,
\begin{equation}\label{eq:5vsystem-2}
\left\{
\begin{array}{ccccl}
a_Y(y,z)&&-d(\xi,z)&=&\langle f_Y,z\rangle +c(Br,z),
\\
&&c(\xi,Bq)&=& 0
\\
-d(y,\eta)&+c(\eta,Bp)&&=&0.
\end{array}
\right.
\end{equation}
\item with $y$ solved out, find $u\in R$, such that
\begin{equation}\label{eq:5vsystem-3}
c(Bu,Bq)=c(y,Bs),\quad\forall\,s\in R.
\end{equation}
\end{enumerate}
\item The problem \eqref{eq:vpexp} can be decoupled in accordance with the problem \eqref{eq:dualdecomp} as 
\begin{enumerate}
\item find $r\in R$, such that
\begin{equation}\label{eq:decom-1}
c(Br,Bv)=\langle f_R,v\rangle,\quad\forall\,v\in R;
\end{equation}
\item with $r\in R$ solved out, find $(y,\xi)\in Y\times \Sigma$, such that for $(z,\eta)\in Y\times \Sigma$,
\begin{equation}\label{eq:decom-2}
\left\{
\begin{array}{cccl}
a_R(y,z)&-d(\xi,z)&=&\langle f_Y,z\rangle+c(Br,z),
\\
-d(y,\eta)&&=&0;
\end{array}
\right.
\end{equation}
\item with $y$ solved out, find $u\in R$, such that
\begin{equation}\label{eq:decom-3}
c(Bu,Bs)=c(y,Bs), \quad \forall\,s\in R.
\end{equation}
\end{enumerate}
\end{enumerate}
\end{theorem}

\subsection{Discretisation of the decomposed formulation}
Suppose subspaces $R_{h,i}\subset R$, $i=1,2,3$, $\Sigma_h\subset \Sigma$ and $Y_h\subset Y$ are respectively closed, and $BR_{h,2}\subset \Sigma_h$. For Problem \eqref{eq:ehvardualdecompctod}, we propose the discretisation scheme associated with \eqref{eq:5vsystem-1}--\eqref{eq:5vsystem-3}:
\begin{enumerate}
\item find $r_h\in R_{h,1}$, such that
\begin{equation}\label{eq:5vsystem-1dis}
c(Br_h,Bv_h)=\langle f_R,v_h\rangle,\quad\forall\,v_h\in R_{h,1};
\end{equation}
\item with $r_h$ solved out, find $(p_h,y_h,\xi_h)\in R_{h,2}\times Y_h\times \Sigma_h$, such that, for $(q_h,z_h,\eta_h)\in R_h\times Y_h\times\Sigma_h$,
\begin{equation}\label{eq:5vsystem-2dis}
\left\{
\begin{array}{ccccl}
a_Y(y_h,z_h)&&-d(\xi_h,z_h)&=&\langle f_Y,z_h\rangle +c(Br_h,z_h),
\\
&&c(\xi_h,Bq_h)&=& 0
\\
-d(y_h,\eta_h)&+c(\eta_h,Bp_h)&&=&0.
\end{array}
\right.
\end{equation}
\item with $y_h$ solved out, find $u_h\in R_{h,3}$, such that
\begin{equation}\label{eq:5vsystem-3dis}
c(Bu_h,Bq_h)=c(y_h,Bs_h),\quad\forall\,s_h\in R_{h,3}.
\end{equation}
\end{enumerate}
Meanwhile, for Problem \eqref{eq:redehvardualdecomp}, we propose a discretisation associated with  \eqref{eq:decom-1} -- \eqref{eq:decom-3}:
\begin{enumerate}
\item find $r_h\in R_{h,1}$, such that
\begin{equation}\label{eq:decom-1dis}
c(Br_h,Bv_h)=\langle f_R,v_h\rangle,\quad\forall\,v_h\in R_{h,1};
\end{equation}
\item with $r_h$ solved out, find $(y_h,\xi_h)\in Y_h\times \Sigma_h$, such that for $(z_h,\eta_h)\in Y_h\times \Sigma_h$,
\begin{equation}\label{eq:decom-2dis}
\left\{
\begin{array}{cccl}
a_Y(y_h,z_h)&-d_h(\xi_h,z_h)&=&\langle f_Y,z_h\rangle+c(Br_h,z_h),
\\
-d_h(y_h,\eta_h)&&=&0;
\end{array}
\right.
\end{equation}
\item with $y_h$ solved out, find $u_h\in R_{h,3}$, such that
\begin{equation}\label{eq:decom-3dis}
c(Bu_h,Bs_h)=c(y_h,Bs_h),\quad\forall\,s_h\in R_{h,3}.
\end{equation}
\end{enumerate}
By \textbf{Hypothesis 0} and \textbf{II}, problems \eqref{eq:5vsystem-1dis} and \eqref{eq:5vsystem-3dis} are well-posed. For \eqref{eq:5vsystem-2dis}, we need Property III:
\paragraph{\bf Property III} 
\begin{enumerate}
\item $a_Y(\cdot,\cdot)$ is coercive on $\{z_h\in Y_h:d(z_h,\eta_h)=0,\forall\,\eta_h\in (BR_{h,2})^\perp\}$;
\item $\displaystyle\inf_{\tau_h\in (BR_{h,2})^\perp}\sup_{y_h\in Y_h}\frac{d(\tau_h,y_h)}{\|\tau_h\|_\Sigma\|y_h\|_Y}\geqslant C$ where $(BR_{h,2})^\perp$ is the orthogonal complement of $BR_{h,2}$ in $\Sigma_h$.
\end{enumerate}

Similar to Lemma \ref{lem:infsupequi}, we can prove the lemma below.
\begin{lemma}
$\displaystyle\inf_{\tau_h\in (BR_{h,2})^\perp}\sup_{y_h\in Y_h}\frac{d(\tau_h,y_h)}{\|\tau_h\|_\Sigma\|y_h\|_Y}
\cequiv \displaystyle\inf_{\tau_h\in \Sigma_h}\sup_{y_h\in Y_h,r_h\in R_{h,2}}\frac{d(\tau_h,y_h)+c(\tau_h,Br_h)}{\|\tau_h\|_\Sigma(\|y_h\|_Y+\|Br_h\|_H)}.$
\end{lemma}

\begin{lemma}
If $R_{h,2}$, $Y_h$ and $\Sigma_h$ satisfy {\bf Property III},  the problem  \eqref{eq:5vsystem-2dis}  is well posed.
\end{lemma}
\begin{lemma}\label{lem:absest}
Given $F_R\in R'$ and $F_Y\in Y'$, let $r\sim(p,y,\xi)\sim u$ and $r_h\sim(p_h,y_h,\xi_h)\sim u_h$ be the solution of Problems \eqref{eq:5vsystem-1}--\eqref{eq:5vsystem-3} and Problems \eqref{eq:5vsystem-1dis}--\eqref{eq:5vsystem-3dis}, respectively. Then
\begin{enumerate}
\item $\displaystyle\|r-r_h\|_{R}\leqslant C \inf_{v_h\in R_{h,1}}\|r-s_h\|_R$;
\item $p=p_h=0$;
\item $\displaystyle\|y-y_h\|_{Y}+\|\xi-\xi_h\|_\Sigma\leqslant C\inf_{z_h\in Y_h,\eta_h\in \Sigma_h}\|y-z_h\|_{Y}+\|\xi-\eta_h\|_\Sigma+\sup_{z_h\in Y_h}\frac{c(B(r_h-r),z_h)}{\|z_h\|_Y}$;
\item $\displaystyle\|u-u_h\|_{R}\leqslant C( \inf_{v_h\in R_{h,3}}\|u-v_h\|_R+\sup_{s_h\in R_{h,3}}\frac{c(y-y_h,Bs_h)}{\|s_h\|_R})$.
\end{enumerate}
\end{lemma}

\begin{lemma}
Given $F_R\in R'$ and $F_Y\in Y'$, let $r_h\sim(p_h,y_h,\xi_h)\sim u_h$ be the solution of Problems \eqref{eq:5vsystem-1dis}--\eqref{eq:5vsystem-3dis}. Then $r_h\sim(y_h,\xi_h)\sim u_h$ is a solution of Problems \eqref{eq:decom-1dis}--\eqref{eq:decom-3dis}; moreover, if $\hat r_h\sim(\hat y_h,\hat\xi_h)\sim\hat u_h$ is another solution of Problems \eqref{eq:decom-1dis}--\eqref{eq:decom-3dis}, then $\hat r_h=r_h$, $\hat y_h=y_h$ and $\hat u_h=u_h$.
\end{lemma}


%
%
%
\section{Decoupled formulation of the 3D bi-Laplacian equation}
\label{sec:mixbl}

In this section, we study the finite element method for the three dimensional bi-Laplacian equation
\begin{equation}\label{eq:modelbl}
(-\Delta)\left(-\alpha(x)\Delta u\right)=f,
\end{equation}
with the homogeneous boundary data $u=0$ and $\nabla u=\undertilde{0}$ and with $0<\alpha_s<\alpha(x)<\alpha_b$.

The primal variational formulation is, given $f\in H^{-2}(\Omega)$, to find $u\in H^2_0(\Omega)$, such that
\begin{equation}
(\alpha\Delta u,\Delta v)=\langle f,v\rangle,\quad\forall\,v\in H^2_0(\Omega).
\end{equation}

As $H^{2}_0(\Omega)$ is configurated by $\{H^1_0(\Omega),\undertilde{H}{}^1_0(\Omega),\nabla \}$, we rewrite the model problem as, given $f_1\in H^{-1}(\Omega)$ and $\uf{}_2\in \undertilde{H}{}^{-1}(\Omega)$, to find $u\in H^2_0(\Omega)$, such that
\begin{equation}\label{eq:primalbl}
(\alpha\dv\nabla u,\dv\nabla v)+(\curl\nabla u,\curl\nabla v)=\langle f_1,v\rangle+\langle\uf{}_2,\nabla v\rangle,\quad\forall\,v\in H^2_0(\Omega).
\end{equation}
Note that we add a mute term $(\curl\nabla u,\curl\nabla v)$ here without any difference.

\subsection{Structure of Sobolev spaces}

\begin{lemma}\label{lem:capreg}
(\cite{GiraultRaviart1986})For polyhedron domain $\Omega$, $H_0(\curl,\Omega)\cap H_0(\dv,\Omega)=\undertilde{H}{}^1_0(\Omega).$ Moreover, if $\Omega$ is convex, $H_0(\curl,\Omega)\cap H(\dv,\Omega)\subset \undertilde{H}^1(\Omega),$ and $H(\curl,\Omega)\cap H_0(\dv,\Omega)\subset \undertilde{H}^1(\Omega).$
\end{lemma}

\begin{lemma}\label{lem:curlsur} (\cite{Pasciak.J;Zhao.J2002,Hiptmair.R;Xu.J2007})
Given $\ueta\in H_0(\curl,\Omega)$, there exists a $\uphi\in\undertilde{H}{}^1_0(\Omega)$, such that $\curl\uphi=\curl\ueta$, and $\|\uphi\|_{1,\Omega}\leqslant C\|\ueta\|_{\curl,\Omega}$, with $C$ a generic positive constant uniform in $H_0(\curl,\Omega)$.
\end{lemma}
\begin{lemma}\label{lem:ficurl}(Friedrichs Inequality, c.f. \cite{Boffi.D;Brezzi.F;Fortin.M2013})
There exists a constant $C$, such that it holds for $\uv\in N_0(\curl,\Omega)$ that
\begin{equation}
\|\uv\|_{0,\Omega}\leqslant C\|\curl\uv\|_{0,\Omega}.
\end{equation}
\end{lemma}
\begin{lemma}(regular decomposition of $H_0(\curl,\Omega)$)
Given $\ueta\in H_0(\curl,\Omega)$, there exists a $w\in H^1_0(\Omega)$ and $\uphi\in\undertilde{H}{}^1_0(\Omega)$, such that 
$\ueta=\nabla w+\uphi$ and $\|\ueta\|_{\curl,\Omega}\geqslant C(\|w\|_{1,\Omega}+\|\uphi\|_{1,\Omega})$.
\end{lemma}
\begin{proposition}
\begin{equation}\label{eq:infsuphhdiv}
\inf_{\ueta\in N_0(\curl,\Omega)\setminus\{\mathbf{0}\}}\sup_{\uphi\in\undertilde{H}{}^1_0(\Omega)}\frac{(\curl\uphi,\curl\ueta)}{\|\uphi\|_{1,\Omega}\|\curl\ueta\|_{0,\Omega}}
=
\inf_{\utau\in\mathring{H}_0(\dv,\Omega)\setminus\{\mathbf{0}\}}\sup_{\uphi\in\undertilde{H}{}^1_0(\Omega)}\frac{(\curl\uphi,\utau)}{\|\uphi\|_{1,\Omega}\|\utau\|_{0,\Omega}}\geqslant C>0.
\end{equation}
\end{proposition}
\subsection{Decoupled formulations}
According to Section \ref{subsec:decform}, we can present the decoupled formulations of \eqref{eq:primalbl} as below. 
\paragraph{ Decouple formulation \bf A} \eqref{eq:primalbl} can be decoupled to the system below:
\begin{enumerate}
\item find $r\in H^1_0(\Omega)$, such that
\begin{equation}\label{eq:probI-1}
(\nabla r,\nabla v)=-\langle f_1,v\rangle,\forall\,v\in H^1_0(\Omega);
\end{equation}
\item with $r$ solved out, find $(\uphi,\uzeta,p)\in \undertilde{H}{}^1_0(\Omega)\times H_0(\curl,\Omega)\times H^1_0(\Omega)$, such that, for $(\upsi,\ueta,q)\in  \undertilde{H}{}^1_0(\Omega)\times H_0(\curl,\Omega)\times H^1_0(\Omega)$,
\begin{equation}\label{eq:probI-2}
\left\{
\begin{array}{ccclll}
(\alpha\dv\uphi,\dv\upsi)+(\curl\uphi,\curl\upsi)&&(\curl\uzeta,\curl\upsi)&=&\langle\uf{}_2,\upsi\rangle-(\nabla r,\upsi),&
\\
&&(\uzeta,\nabla q) &=&0
\\
(\curl\uphi,\curl\ueta)&+(\ueta,\nabla p)&&=&0;&
\end{array}
\right.
\end{equation}
\item with $\uphi$ solved out, find $u\in H^1_0(\Omega)$, such that
\begin{equation}\label{eq:probI-3}
(\nabla u,\nabla s)=(\uphi,\nabla s),\ \forall\,s\in H^1_0(\Omega).
\end{equation}
\end{enumerate}

\paragraph{Decoupled formulation \bf B} \eqref{eq:primalbl} can be decoupled to the system below:
\begin{enumerate}
\item find $r\in H^1_0(\Omega)$, such that
\begin{equation}\label{eq:probII-1}
(\nabla r,\nabla v)=-\langle f_1,v\rangle,\forall\,v\in H^1_0(\Omega);
\end{equation}
\item with $r$ solved out, find $(\uphi,\uzeta)\in \undertilde{H}{}^1_0(\Omega)\times H_0(\curl,\Omega)$, such that, for $(\upsi,\ueta)\in  \undertilde{H}{}^1_0(\Omega)\times H_0(\curl,\Omega)$,
\begin{equation}\label{eq:probII-2}
\left\{
\begin{array}{ccclll}
(\alpha\dv\uphi,\dv\upsi)+(\curl\uphi,\curl\upsi)&&(\curl\uzeta,\curl\upsi)&=&\langle\uf{}_2,\upsi\rangle-(\nabla r,\upsi),&
\\
(\curl\uphi,\curl\ueta)&&&=&0;&
\end{array}
\right.
\end{equation}
\item with $\uphi$ solved out, find $u\in H^1_0(\Omega)$, such that
\begin{equation}\label{eq:probII-3}
(\nabla u,\nabla s)=(\uphi,\nabla s),\ \forall\,s\in H^1_0(\Omega).
\end{equation}
\end{enumerate}

\begin{theorem}\label{thm:regconv}
Given $f_1\in H^{-1}(\Omega)$ and $\uf{}_2\in \undertilde{H}{}^{-1}(\Omega)$, Problem \textbf A admits a unique solution $r\sim(\uphi,\uzeta,0)\sim u$, and $u$ solves the problem \eqref{eq:primalbl} and $r\sim(\uphi,\uzeta)\sim u$ solves Problem \textbf{B}.  Further, 
\begin{enumerate}
\item if $\Omega$ is a convex polyhedron, then,
\begin{enumerate}
\item if $f_1\in L^2(\Omega)$, then $\|r\|_{2,\Omega}\cequiv \|f_1\|_{0,\Omega}$;
\item if $\alpha$ is smooth on $\Omega$ and $\uf{}_2\in \undertilde{L}{}^2(\Omega)$, then
\begin{equation}
\|r\|_{1,\Omega}+\|\uzeta\|_{1,\Omega}+\|\curl\uzeta\|_{1,\Omega}+\|\uphi\|_{2,\Omega}+\|u\|_{3,\Omega}\cequiv \|f_1\|_{-1,\Omega}+\|\uf{}_2\|_{0,\Omega};
\end{equation}
\end{enumerate}
\item if $\hat r\sim(\hat\uphi,\hat\uzeta)\sim\hat u$ is another solution of Problem \textbf{B}, then $\hat r=r$, $\hat\uphi=\uphi$ and $\hat u=u$.
\end{enumerate}
\end{theorem}

\begin{proof}
We only have to show the regularity estimate 1(b). Given $f_1\in H^{-1}(\Omega)$ and $\uf{}_2\in\undertilde{L}^2(\Omega)$, by the regularity of fourth order problem, $\|u\|_{3,\Omega}\cequiv \|\uphi\|_{2,\Omega}\lesssim \|f_1\|_{-1,\Omega}+\|\uf{}_2\|_{0,\Omega}$. As $(\uzeta,\nabla q)=0$ for any $q\in H^1_0(\Omega)$, we have $\dv\uzeta=0$, and thus $\uzeta\in H_0(\curl,\Omega)\cap H(\dv,\Omega)\subset \undertilde{H}{}^1(\Omega)$. Besides, $(\curl\uzeta,\curl\upsi)=(\uf{}_2-\nabla\alpha\nabla\cdot \uphi-\nabla r,\upsi)$ for any $\upsi\in\undertilde{H}{}^1_0(\Omega)$, we thus obtain $\curl\curl\uzeta=\uf{}_2-\nabla\alpha\nabla\cdot \uphi-\nabla r\in \undertilde{L}{}^2(\Omega)$. Since $\curl\uzeta\subset H_0(\dv,\Omega)$, we obtain further $\curl\uzeta\in H(\curl,\Omega)\cap H_0(\dv,\Omega)\subset\undertilde{H}{}^1(\Omega)$. The remaining of the lemma follows by the regularity of Poisson equation. The proof is completed.
\end{proof}


%
%
%
\section{Finite element discretisation}
\label{sec:fem}

In this section, we choose respective finite element subspaces of $H^1_0(\Omega)$, $\undertilde{H}{}^1_0(\Omega)$ and $H_0(\curl,\Omega)$ and use them to replace the Sobolev spaces in the mixed system to generate a discretisation. 

\subsection{Subdivision and finite elements}

For $K$ a tetrahedron with $a_i$, $i=1:4$, the vertices, denote by  $F_i$ the corresponding opposite faces. The barycentre coordinates are denoted as usual by $\lambda_i$, $i=1,2,3,4$.
Denote $q_0=\lambda_1\lambda_2\lambda_3\lambda_4$, and $q_i=q_0/\lambda_i$, $i=1,2,3,4$. Obviously, $q_0$ vanishes on the faces of $K$, and $q_i$ vanishes on the faces of $K$ other than $F_i$. As usual, we use $P_k(F)$ for the set of polynomials on a face $F$ of degrees not higher than $k$, $P_k(K)$ for the set of polynomials on $K$ of degrees not higher than $k$, and $\hat{P}_{k}(K)$ is the space of homogeneous $k$-th degree polynomials. Besides, denote shape function spaces:
\begin{itemize}
\item $
P_k^{i}(K):={\rm span}\{\lambda_1^{\alpha_1}\lambda_2^{\alpha_2}\lambda_3^{\alpha_3}\lambda_4^{\alpha_4}:\alpha_i=0,\ \sum_{j=1}^4\alpha_j=k\}
$
for any 4-index $(\alpha_1,\alpha_2,\alpha_3,\alpha_4)$;

\item $BF_{k}^i(K):=\{pq_i:p\in P_{k}^i(K)\}$, $BB_k(K):=\{pq_0:p\in P_k(K)\}$;
\item $\mathfrak{P}_k(K):=P_k(K)+\sum_{i=1}^4BF_{k-1}^i(K)+BB_{k-2}(K)$;
\item $\mathbb{E}_k(K):=\{\uu+\uv\times \ux:\uu\in (P_{k-1}(K))^3,\ \uv\in (\hat{P}_{k-1}(K))^3\}$.
\end{itemize}
Evidently, $BF_{k}^i(K)\subset P_{k+3}(K)$, $BB_k(K)\subset P_{k+4}(K)$, $P_k^i(K)\subsetneq P_k(K)$, and $P_k^i(K)|_{F_i}=P_k(F_i)=P_k(K)|_{F_i}$. 

Let $\mathcal{G}_h$ be a tetrahedron subdivision of $\Omega$, such that $\bar\Omega=\cup_{K\in\mathcal{G}_h}\bar K$. Denote by $\mathcal{F}_h$, $\mathcal{F}_h^i$, $\mathcal{E}_h$, $\mathcal{E}_h^i$, $\mathcal{X}_h$ and $\mathcal{X}_h^i$ respectively the set of faces, interior faces, edges, interior edges, vertices and interior vertices.  Denote by $P_k(\mathcal{G}_h)$ the space of piecewise $k$-th degree polynomials on $\mathcal{G}_h$.
Associatedly, define finite element spaces by
\begin{itemize}
\item $L_h^k:=\{w\in H^1(\Omega):w|_K\in P_k(K),\ \forall\,K\in \mathcal{G}_h\}$, and $L_{h0}^k=L_h^k\cap H^1_0(\Omega)$;

\item $\mathfrak{L}_h^k:=\{w\in H^1(\Omega):w|_K\in \mathfrak{P}_k(K),\ \forall\,K\in \mathcal{G}_h\}$, and $\mathfrak{L}^k_{h0}=\mathfrak{L}^k_h\cap H^1_0(\Omega)$;
\item  $\mathbb{N}_h^k:=\{\uw\in H(\curl,\Omega):\uw|_K\in \mathbb{E}_k(K),\ \forall\,K\in \mathcal{G}_h\}$, and $\mathbb{N}_{h0}^k=\mathbb{N}_h^k\cap H_0(\curl,\Omega)$.
\end{itemize}

It is well known that (\cite{Nedelec1980}) $\{\ueta\in\mathbb{N}_{h0}^k:\curl\ueta=\undertilde{0}\}=\nabla L^k_{h0}.$ Thus denote $N^k_{h0}(\curl):=\{\ueta\in \mathbb{N}^k_{h0}:(\ueta,\nabla v)=0,\ \forall\,v\in L^k_{h0}\}$. Namely, $ N_{h0}^k(\curl)$ is the orthogonal completion of $(\nabla L^k_{h0})$ in $\mathbb{N}^k_{h0}$ in $L^2$ inner product and also $H(\curl,\Omega)$ inner product.

\begin{lemma}\label{lem:ficurldis}(Discretized Friedrichs Inequality, c.f. \cite{Kikuchi.F1989,Hiptmair.R2002,AFW2006})
There exists a constant $C$, such that $\|\uv{}_h\|_{0,\Omega}\leqslant C\|\curl\uv{}_h\|_{0,\Omega}$ for $\uv{}_h\in \mathbb{N}_{h0}^k(\curl)$.
\end{lemma}

\begin{lemma}\label{lem:infsupdis}
For any integer $k$, there is a constant $C>0$, such that
\begin{equation}
\inf_{\undertilde{\eta}{}_h\in N^k_{h0}(\curl)\setminus\{\undertilde{0}\}}\sup_{\uphi{}_h\in (\mathfrak{L}^k_{h0})^3}\frac{(\curl\uphi{}_h,\curl\ueta{}_h)}{\|\uphi{}_h\|_{1,\Omega}\|\ueta{}_h\|_{\curl,\Omega}}\geqslant C.
\end{equation}
\end{lemma}
\begin{proof}
Evidently, $\curl \mathbb{N}^k_{h0}\subset \mathring H_0(\dv,\Omega)\cap (P_{k-1}(\mathcal{T}_h))^3$, thus it suffices for us to prove:
\begin{equation}
\inf_{\undertilde{q}{}_h\in\mathring H_0(\dv,\Omega)\cap (P_{k-1}(\mathcal{T}_h))^3\setminus\{\undertilde{0}\}}\sup_{\uphi{}_h\in (\mathfrak{L}^k_{h0})^3}\frac{(\curl\uphi{}_h,\uq{}_h)}{\|\uphi{}_h\|_{1,\Omega}\|\uq{}_h\|_{0,\Omega}}\geqslant C.
\end{equation}

Noting \eqref{eq:infsuphhdiv}, we try to construct a Frotin operator. Denote by $\Pi_1:H^1_0(\Omega)\to L_{h0}^1$ the Cl\'ement interpolant (\cite{Clement.P1975}), and define $\Pi_h: {H}{}^1_{0}(\Omega)\to \mathfrak{L}_{h0}^k$ such that
\begin{equation}
\left\{
\begin{array}{ll}
\displaystyle(\Pi_h v)(M)=(\Pi_1 v)(M), & \quad\forall\, M\in\mathcal{X}_h;
\\
\displaystyle\int_F(\Pi_h v)\gamma ds=\int_F v\gamma ds, &\quad\forall\,\gamma\in P_{k-1}(F)\ \mbox{and}\ \forall\,F\in\mathcal{F}_h;
\\
\displaystyle\int_K(\Pi_h v)\delta dx = \int_K v\delta dx, & \quad\forall\,\delta\in P_{k-2}(K)\ \mbox{and}\ \forall\,K\in \mathcal{G}_h.
\end{array}
\right.
\end{equation}
Note that $\mathfrak{L}^k_{h0}$ is $L^k_{h0}$ combined with body and face bubbles, and the operator $\Pi_h$ is well-defined. Then define $\undertilde{\Pi}{}_h:\undertilde{H}{}^1_0(\Omega)\to (\mathfrak{L}^k_{h0})^3$ by
$$
\undertilde{\Pi}{}_h\uv=(\Pi_h v_1,\Pi_h v_2,\Pi_h v_3)^\top,\ \ \mbox{for}\ \,\uv=(v_1,v_2,v_3)^\top.
$$
By standard technique, we obtain
$$
\sum_{K}h_K^{2r-2}|v-\Pi_h v|_{r,K}^2\leqslant C\|v\|_{1,\Omega}^2,\ \ r=0,1,
$$
and
$$
\sum_{K}h_K^{2r-2}|\uv-\undertilde{\Pi}{}_h \uv|_{r,K}^2\leqslant C\|\uv\|_{1,\Omega}^2,\ \ r=0,1.
$$
Now, as $\undertilde{q}{}_h\in(P_{k-1}(\mathcal{T}_h))^3$ and thus $\curl\,\undertilde{q}{}_h\in (P_{k-2}(\mathcal{T}_h))^3$, we have
\begin{multline}
\int_\Omega \curl(\uv-\undertilde{\Pi}{}_h\uv)\cdot \uq{}_h=\sum_{K\in \mathcal{G}_h}\int_K(\curl(\uv-\undertilde{\Pi}{}_h\uv)\cdot \uq{}_h)
\\
=\sum_{K\in\mathcal{G}_h}\left(\int_K(\uv-\undertilde{\Pi}{}_h\uv)\cdot\curl\uq{}_h+\sum_{F\subset\partial K}\int_F(\uv-\undertilde{\Pi}{}_h\uv)\times\mathbf{n}_F\cdot\uq{}_h\right)=0.
\end{multline}
Thus $\undertilde{\Pi}{}_h$ is a Fortin operator such that
\begin{equation}
\|\undertilde{\Pi}{}_h\uv\|_{1,\Omega}\leqslant C\|\uv\|_{1,\Omega},\ (\curl(\uv-\undertilde{\Pi}{}_h\uv),\uq{}_h)=0,\ \forall\,\uq{}_h\in (P_{k-1}(\mathcal{T}_h))^3,
\end{equation}
and this completes the proof.
\end{proof}

\subsection{Finite element discretization of \eqref{eq:probI-1}$\sim$\eqref{eq:probI-3}}

Associated with the formulation \eqref{eq:probI-1}$\sim$\eqref{eq:probI-3}, we propose a discretization scheme for \eqref{eq:primalbl} below.
\paragraph{Discretization \bf A}
\begin{enumerate}
\item find $r_h\in L^k_{h0}$, such that
\begin{equation}\label{eq:probI-1dis}
(\nabla r_h,\nabla v_h)=-\langle f_1,v_h\rangle,\forall\, v_h\in L^k_{h0};
\end{equation}
\item find $(\uphi{}_h,\uzeta{}_h,p_h)\in (\mathfrak{L}^m_{h0})^3\times \mathbb{N}^k_{h0}\times L^k_{h0}$, such that, for $(\upsi{}_h,\ueta{}_h,q_h)\in  (\mathfrak{L}^m_{h0})^3\times \mathbb{N}^k_{h0}\times L^k_{h0}$,
\begin{equation}\label{eq:probI-2dis}
\left\{
\begin{array}{ccclll}
(\alpha\dv\uphi{}_h,\dv\upsi{}_h)+(\curl\uphi{}_h,\curl\upsi{}_h)&&(\curl\uzeta{}_h,\curl\upsi{}_h)&=&\langle\uf{}_2,\upsi{}_h\rangle-(\nabla r_h,\upsi{}_h),&
\\
&&(\uzeta{}_h,\nabla q_h) &=&0,
\\
(\curl\uphi{}_h,\curl\ueta{}_h)&+(\ueta{}_h,\nabla p_h)&&=&0;&
\end{array}
\right.
\end{equation}
\item find $u_h\in L^m_{h0}$, such that
\begin{equation}\label{eq:probI-3dis}
(\nabla u_h,\nabla s_h)=(\uphi{}_h,\nabla s_h),\ \forall\,s_h\in L^m_{h0}.
\end{equation}
\end{enumerate}

By Lemma \ref{lem:infsupdis}, if $m\geqslant k$, all the three subproblems are well-posed. The estimate Lemma \ref{lem:absest} works. 

\subsubsection{An economical optimal finite element quintet on convex polyhedrons}
By Lemma \ref{lem:absest}, the rate of $\|u-u_h\|_{1,\Omega}$ can be comparable to $\|\uphi-\uphi{}_h\|_{0,\Omega}$, and $\|\uphi-\uphi{}_h\|_{1,\Omega}$ can be comparable to $\|r-r_h\|_{0,\Omega}$; this is coincident with the regularity estimate of Theorem \ref{thm:regconv} on convex domains. An economical optimal scheme is to set $k=1$ and $m=2$ in \eqref{eq:probI-1dis}$\sim$\eqref{eq:probI-3dis}.
\begin{lemma}
Let $\Omega$ be a convex polyhedron, and assume $\alpha$ is smooth on $\Omega$. Let $f_1\in L^2(\Omega)$ and $\uf{}_2\in \undertilde{L}^2(\Omega)$. Let $r\sim(\uphi,\uzeta,p)\sim u$ and $r_h\sim(\uphi{}_h,\uzeta{}_h,p_h)\sim u_h$ be the solutions of Decoupled formulation \textbf{A} and Discretization \textbf{A}, respectively. Then
\begin{enumerate}
\item $\|r-r_h\|_{1,\Omega}\leqslant Ch\|f_1\|_{0,\Omega}$ and $\|r-r_h\|_{0,\Omega}\leqslant Ch^2\|f_1\|_{0,\Omega}$;
\item $\|\uphi-\uphi{}_h\|_{1,\Omega}+\|\uzeta-\uzeta{}_h\|_{\curl,\Omega}+\|p-p_h\|_{1,\Omega}\leqslant Ch(\|f_1\|_{-1,\Omega}+\|\uf{}_2\|_{0,\Omega})$;
\item $\|\uphi-\uphi{}_h\|_{0,\Omega}\leqslant Ch^2(\|f_1\|_{0,\Omega}+\|\uf{}_2\|_{0,\Omega})$ and $\|u-u_h\|_{1,\Omega}\leqslant Ch^2(\|f_1\|_{0,\Omega}+\|\uf{}_2\|_{0,\Omega})$.
\end{enumerate}
\end{lemma}
\begin{proof}
We only have to estimate $\|\uphi-\uphi{}_h\|_{0,\Omega}$ and $\|u-u_h\|_{1,\Omega}$. The others follow from Lemma \ref{lem:absest} and Theorem \ref{thm:regconv} directly. 

Denote $\undertilde{V}:=H^{1}_{0}(\Omega)\times\undertilde{H}{}^1_{0}(\Omega)\times H^{1}_{0}(\Omega)\times H_{0}(\curl,\Omega)\times H^{1}_{0}(\Omega)$, and define $a_s(\cdot,\cdot)$ on $\undertilde{V}$ as
\begin{multline*}
a_s((u,\uphi,p,\uzeta,r),(v,\upsi,q,\ueta,s)):=
-(\nabla r,\nabla v)+(\alpha\dv\uphi,\dv\upsi)+(\curl\uphi,\curl\upsi)
\\
+(\curl\uzeta,\curl\upsi)
+(\nabla r,\upsi)+(\uzeta,\nabla q)+(\curl\uphi,\curl\ueta)+(\ueta,\nabla p)-(\nabla u,\nabla s)+(\uphi,\nabla s).
\end{multline*}
Decoupled formulation \textbf{A} is equivalent to, with same variables, finding $(u,\uphi,p,\uzeta,r)\in \undertilde{V}$, such that 
\begin{equation}
a_s((u,\uphi,p,\uzeta,r),(v,\upsi,q,\ueta,s))=\langle f_1,v\rangle+\langle \uf{}_2,\upsi\rangle,\ \forall\,(v,\upsi,q,\ueta,s)\in \undertilde{V}.
\end{equation}
Denote $\undertilde{V}{}_h:=L^2_{h0}\times (\mathfrak{L}_{h0}^1)^3\times L^1_{h0}\times \mathbb{N}_{h0}^1\times L^1_{h0}$ and $\undertilde{V}{}_h':=L^1_{h0}\times (\mathfrak{L}_{h0}^1)^3\times L^1_{h0}\times \mathbb{N}_{h0}^1\times L^2_{h0}$, and Discretization \textbf{A} with $m=2$ and $k=1$ is equivalent to finding $(u_h,\uphi{}_h,p_h,\uzeta{}_h,r_h)\in \undertilde{V}{}_h$, such that 
\begin{equation}
a_s((u_h,\uphi{}_h,p_h,\uzeta{}_h,r_h),(v_h,\upsi{}_h,q_h,\ueta{}_h,s_h))=\langle f_1,v_h\rangle+\langle \uf{}_2,\upsi{}_h\rangle,\ \forall\,(v_h,\upsi{}_h,q_h,\ueta{}_h,s_h)\in \undertilde{V}{}_h'.
\end{equation}
Now denote $\hat{\undertilde{V}}{}_h:=L^1_{h0}\times (\mathfrak{L}_{h0}^1)^3\times L^1_{h0}\times \mathbb{N}_{h0}^1\times L^1_{h0}$, and let $(\hat u_h,\hat\uphi{}_h,\hat p_h,\hat\uzeta{}_h,\hat r_h)\in \hat{\undertilde{V}}{}_h$ be such that
\begin{equation}
a_s((\hat u_h,\hat\uphi{}_h,\hat p_h,\hat\uzeta{}_h,\hat r_h),(\hat v_h,\hat\upsi{}_h,\hat q_h,\hat\ueta{}_h,\hat s_h))= \langle f_1,v_h\rangle+\langle \uf{}_2,\upsi{}_h\rangle \forall\,(\hat v_h,\hat\upsi{}_h,\hat q_h,\hat\ueta{}_h,\hat s_h)\in \hat{\undertilde{V}}{}_h.
\end{equation}
Then it follows that
$$
\hat r_h=r_h,\ \ \hat \uzeta{}_h=\uzeta{}_h,\ \ \hat p_h=p_h,\ \ \mbox{and}\ \ \hat\uphi{}_h=\uphi{}_h.
$$
By Ce\'a lemma,
\begin{multline*}
\|r-\hat{r}_h\|_{1,\Omega}+\|\uzeta-\hat\uzeta{}_h\|_{\curl,\Omega}+\|p-\hat p_h\|_{1,\Omega}+\|\uphi-\hat \uphi{}_h\|_{1,\Omega}+\|u-\hat{u}_h\|_{1,\Omega}
\\
\leqslant Ch(\|r\|_{2,\Omega}+\|\uzeta\|_{1,\Omega}+\|\curl\uzeta\|_{1,\Omega}+\|p\|_{2,\Omega}+\|\uphi\|_{2,\Omega}+\|u\|_{2,\Omega})\leqslant Ch(\|f_1\|_{0,\Omega}+\|\uf{}_2\|_{0,\Omega}).
\end{multline*}
Let $(\tilde u,\tilde\uphi,\tilde p,\tilde\uzeta,\tilde r)\in \undertilde{V}$ be the solution of
\begin{equation}\label{eq:auxeqn}
a_s((\tilde u,\tilde\uphi,\tilde p,\tilde\uzeta,\tilde r),( v,\upsi, q,\ueta, s))= (\uphi-\hat\uphi{}_h,\upsi),\ \forall\,( v,\upsi, q,\ueta, s)\in \undertilde{V}.
\end{equation}
Then it can be proved that $\tilde r=0$, $\tilde p=0$, and
$$
\|\tilde\uphi\|_{2,\Omega}+\|\tilde\uzeta\|_{1,\Omega}+\|\curl\tilde\uzeta\|_{1,\Omega}+\|\tilde u\|_{3,\Omega}\leqslant C\|\uphi-\hat\uphi{}_h\|_{0,\Omega}.
$$
By \eqref{eq:auxeqn},
$$
(\uphi-\hat\uphi{}_h,\uphi-\hat\uphi{}_h)=a_s((\tilde u,\tilde\uphi,\tilde p,\tilde\uzeta,\tilde r),( u-\hat u_h,\uphi-\hat\uphi{}_h, p-\hat p_h,\uzeta-\hat\uzeta{}_h, r-\hat r_h)),
$$
and further
$$
(\uphi-\hat\uphi{}_h,\uphi-\hat\uphi{}_h)=a_s((\tilde u-\tilde{v}_h,\tilde\uphi-\tilde\upsi{}_h,\tilde p-\tilde q_h,\tilde\uzeta-\tilde \ueta{}_h,\tilde r-\tilde s_h),( u-\hat u_h,\uphi-\hat\uphi{}_h, p-\hat p_h,\uzeta-\hat\uzeta{}_h, r-\hat r_h)),
$$
for any $(\tilde{v}_h,\tilde\upsi{}_h,\tilde q_h,\tilde \ueta{}_h,\tilde s_h)\in \hat{\undertilde{V}}{}_h$. Thus
\begin{multline*}
\|\uphi-\hat\uphi{}_h\|_{0,\Omega}^2\leqslant Ch^2(\|\tilde\uphi\|_{2,\Omega}+\|\tilde\uzeta\|_{1,\Omega}+\|\curl\tilde\uzeta\|_{1,\Omega}+\|\tilde u\|_{2,\Omega})
\\
\cdot(\|r\|_{2,\Omega}+\|\uzeta\|_{1,\Omega}+\|\curl\uzeta\|_{1,\Omega}+\|p\|_{2,\Omega}+\|\uphi\|_{2,\Omega}+\|u\|_{2,\Omega}),
\end{multline*}
which leads further to
$$
\|\uphi-\hat\uphi{}_h\|_{0,\Omega}\leqslant  Ch^2(\|f_1\|_{0,\Omega}+\|\uf{}_2\|_{0,\Omega}).
$$
By Lemma \ref{lem:absest}, this  leads to $\|u-u_h\|_{1,\Omega}\leqslant Ch^2(\|f_1\|_{0,\Omega}+\|\uf{}_2\|_{0,\Omega})$, and completes the proof.
\end{proof}

\subsection{Finite element discretization of \eqref{eq:probII-1}$\sim$\eqref{eq:probII-3}}

Associated with the formulation \eqref{eq:probII-1}$\sim$\eqref{eq:probII-3}, we propose a discretization scheme for \eqref{eq:primalbl} below.
\paragraph{Discretization \bf B}
\begin{enumerate}
\item find $r_h\in L^k_{h0}$, such that
\begin{equation}\label{eq:probII-1}
(\nabla r_h,\nabla v_h)=-\langle f_1,v_h\rangle,\forall\,v_h\in L^k_{h0};
\end{equation}
\item with $r_h$ solved out, find $(\uphi{}_h,\uzeta{}_h)\in (\mathfrak{L}^m_{h0})^3\times \mathbb{N}^k_{h0}$, such that, for $(\upsi{}_h,\ueta{}_h)\in (\mathfrak{L}^m_{h0})^3\times \mathbb{N}^k_{h0}$,
\begin{equation}\label{eq:probII-2}
\left\{
\begin{array}{ccclll}
(\alpha\dv\uphi{}_h,\dv\upsi{}_h)+(\curl\uphi{}_h,\curl\upsi{}_h)&&(\curl\uzeta{}_h,\curl\upsi{}_h)&=&\langle\uf{}_2,\upsi{}_h\rangle-(\nabla r_h,\upsi{}_h),&
\\
(\curl\uphi{}_h,\curl\ueta{}_h)&&&=&0;&
\end{array}
\right.
\end{equation}
\item with $\uphi{}_h$ solved out, find $u_h\in L^m_{h0}$, such that
\begin{equation}\label{eq:probII-3}
(\nabla u_h,\nabla s_h)=(\uphi{}_h,\nabla s_h),\ \forall\,s_h\in L^m_{h0}.
\end{equation}
\end{enumerate}

The lemma below is immediate.
\begin{lemma}
Set $m\geqslant k$ for Discretization \textbf{B}. Given $f_1\in H^{-1}(\Omega)$ and $\uf{}_2\in \undertilde{H}{}^{-1}(\Omega)$, let $r_h\sim(\uphi{}_h,\xi{}_h,p_h)\sim u_h$ be the solution of Discretisation \textbf{A}, then $r_h\sim(\uphi{}_h,\xi{}_h)\sim u_h$ is a solution of Discretisation \textbf{B}, and if $\tilde r_h\sim (\tilde\uphi{}_h,\tilde \xi_h)\sim \tilde u_h$ is another solution of Discretisation \textbf{B}, then $\tilde r_h=r_h$, $\tilde \uphi{}_h=\uphi{}_h$ and $\tilde u_h=u_h$.
\end{lemma}

\begin{remark}
In practice, Discretisation \textbf{B} is an ill-posed system with smaller size, but $u_h$ and $\uphi{}_h$ which are more concerned can be approximated well with the scheme. 
\end{remark}

\section{Concluding remarks}
\label{sec:con}

In this paper, the construction of decoupled mixed element scheme for fourth order problems is studied, and a general process is designed in an intrinsic way for problems of certain types. Once some mild conditions are verified, a problem on high-regularity space can be decoupled to subproblems on low-regularity space. This process may bring in bigger flexibility on designing stable discretisation schemes; the schemes can be implemented by various finite element packages. In this paper, the three dimensional biLaplacian equation with constant and variable coefficient are studied under the general framework. In the future, the application of the methodology  for other problems in applications in, e.g., linear elasticity and acoustics, and other fields, will be discussed. The framework itself can be generalised to more complicated situations.

Due to the equivalence between the primal and decoupled formulations, it could be natural to expect equivalence between some discretised problem in order reduced formulation and some discretised problem in primal formulation (c.f. \cite{Arnold.D;Scott.L;Vogelius.M1988,Arnold.D;Brezzi.F1985,Scott.L;Vogelius.M1985}). Also, their cooperation can be of interests and be discussed in future. Finally we remark that, the framework admits utilisation of ill-posed subproblems; this will bring convenience in designing and implementing numerical schemes and will be discussed more in future. The fast solution of such problems will also be studied.



\end{document}